\documentclass[11pt]{article}
\usepackage{currfile}
\usepackage{color}
\usepackage{mathrsfs}
\usepackage{graphicx}
\usepackage{caption}
\usepackage{float}
\usepackage{subfigure}
\usepackage{tikz}
\usepackage{amsmath,amssymb}
\usepackage[citecolor=blue]{hyperref}

\newtheorem{theorem}{Theorem}[section]

\newtheorem{conjecture}[theorem]{Conjecture}

\newtheorem{lemma}[theorem]{Lemma}

\newtheorem{claim}{Claim}
\newtheorem{fact}{Fact}

\newtheorem{case} {Case}

\newenvironment{proof}{\noindent {\bf Proof.}}{\rule{3mm}{3mm}\par\medskip}
 \newenvironment{theproof}[1][Proof]{\noindent {\bf #1.}}{\rule{3mm}{3mm}\par\medskip}

\date{}

\begin{document}

\title{The minimum  spectral radius of $tP_4$-saturated graphs}

\author{Junxue Zhang$^a$ and Liwen Zhang$^{b}$\thanks{ 
Corresponding author. }\\
\footnotesize $^a$ School of Mathematics and Statistics, Beijing Institute of Technology, Beijing, 100081, China\\ 
\footnotesize $^b$ Center for Combinatorics and LPMC, Nankai University, Tianjin, 300071, China\\
\footnotesize  jxuezhang@163.com, levenzhang512@163.com }
\maketitle

\begin{abstract}
A graph $G$ is called {\em$F$-saturated} if $G$ does not contain $F$ as a subgraph but adding any missing edge to $G$ creates a copy of $F$. 
In this paper, we consider the spectral saturation problem for the linear forest $tP_4$, proving that every $n$-vertex $tP_4$-saturated graph $G$  with $t\geq 2$ and $n\ge 4t$  satisfies $\rho(G)\ge \frac{1+\sqrt{17}}{2}$, and characterizing all $tP_4$-saturated graphs for which equality holds. Moreover, we obtain that, for $t=2$ with odd $n\ge 13 $, and for $t\ge 3$ with $n\ge 6t+4$, the set of $n$-vertex $tP_4$-saturated graphs minimizing the spectral radius is disjoint from that minimizing the number of edges.
\end{abstract}


\section{Introduction}

Let $A(G)$ be the \textit{adjacency matrix} of $G$. The \textit{spectral radius} of $G$ is defined as $\rho(G)=\max\{|\lambda|: \lambda$ is an eigenvalue of $A(G)$\}. Let $F$ be a simple graph. A graph $G$ is \emph{$F$-free} if there is no subgraph of $G$ is isomorphic to $F$. In 1986, Brualdi and Solheid \cite{BS1986} posed the following problem in spectral theory of graphs: Given a set of graphs $\mathcal{G}$, determine $\min\{\rho(G): G\in\mathcal{G} \}$ and $\max\{\rho(G): G\in\mathcal{G} \}$. In 2010, combining this with the classic Tur\'{a}n-type extremal problem (see \cite{T1941}), Nikiforov \cite{NV2010} proposed a Brualdi-Solheid-Tur\'{a}n-type problem: What is the maximum spectral radius among all $F$-free graphs on $n$ vertices? 
Nikiforov \cite{NV2007} showed that if $G$ is a $K_{r+1}$-free graph on $n$ vertices, then $\rho(G)\le \rho(T_{r}(n))$ and equality holds if and only if $G\cong T_r(n)$, where $T_r(n)$ is the Tur\'{a}n graph. 
Subsequently, this problem has been extensively studied for various classes of graphs, such as complete graphs \cite{NV2002,NV2007,W1986}, cycles \cite{LNW2021,NV2008,ZL2023}, and forests \cite{CDT2023,FLSZ2024,NV2010}. Further results can be found in three surveys \cite{CMZX2018, LLF2022,NV2011}. Denote by $\mathbf{EX}(n,F)$ and $\mathbf{EX_{sp}}(n,F)$ the set of $F$-free graphs on $n$ vertices with the maximum number of edges and the maximum spectral radius, respectively. It has been shown in the literature that $\mathbf{EX_{sp}}(n,F) \subseteq \mathbf{EX}(n,F)$ for some graphs $F$. Wang, Kang and Xue \cite{WKX2023} proved that if $F$ is a graph such that each graph in $\mathbf{EX}(n,F)$ is obtained from $T_r(n)$ by adding $O(1)$ edges, then $\mathbf{EX_{sp}}(n,F) \subseteq\mathbf{EX}(n,F)$ holds for sufficiently large $n$, thereby completely solving the conjecture posed by Cioab\u{a}, Desai and Tait \cite{CDT2022}.

A natural counterpart to the Brualdi-Solheid-Tur\'{a}n-type problem is the study of the spectral saturation problem. 
A graph $G$ is {\em $ {F}$-saturated} if $G$ is $F$-free but $G+e$ contains a subgraph isomorphic to $F$ for any $e\in E(\overline{G})$. The minimum number of edges and spectral radius of $F$-saturated graphs on $n$ vertices are respectively denoted by $sat(n, {F})$ and $sat_{sp}(n, {F})$. And denote by $\mathbf{SAT}(n, F )$ (resp. $\mathbf{SAT_{sp}}(n, F)$) the set of ${F}$-saturated graphs on $n$ vertices with size $sat(n, {F})$ (resp. spectral radius $sat_{sp}(n, {F})$. 
Combining the Brualdi-Solheid problem with the saturation problem, 
 Kim, Kim, Kostochka, and O \cite{KKKO2020} investigated the lower bound of $sat_{sp}(n, K_{r+1})$, providing an asymptotically tight lower bound for the spectral radius of any $K_{r+1}$-saturated graph. Later, Kim, Kostochka, O, Shi, and Wang \cite{KKOSW2023} determined $sat_{sp}(n, {K_{r+1}})=\rho(S_{n,r})$ for $r=\{2,3\}$, where $S_{n,r}=K_{r-1}\vee \overline{K}_{n-r+1}$. Recently, Wang and Hou \cite{WH2024} proved $sat_{sp}(n, {K_{r+1}})=\rho(S_{n,r})$ for $r\in \{4,5\}$ and Ai, Liu, O, and Zhang \cite{ALOZ} showed $sat_{sp}(n, {K_{r+1}})=\rho(S_{n,r})$ for $r=4$. Furthermore, in \cite{ALOZ}, they also determined the sharp lower bound on the spectral radius of $tP_{3}$-saturated graphs. Notably, these results suggest that $\mathbf{SAT}(n, {F})\subseteq \mathbf{SAT_{sp}}(n, {F})$ when $F\in \{K_3,K_4,K_5,K_6,tP_3\}$.

 In this paper, we focus on $tP_4$-saturated graphs, where $tP_4$ is the disjoint union of $t$ copies of paths on four vertices. For $t=1$,  K\'{a}szonyi and Tuza \cite{KT86} showed that $sat(n,P_4)=
		\frac{n }{2}$    when $n$ is even, and 
		$sat(n,P_4)=\frac{n+3}{2}$ when $n$ is odd. 
 For $t\ge 2$, 
 Chen, Faudree, Faudree, Gould, Jacobson, and Magnant \cite{CFFGJM2015} raised the conjecture on the saturation number for $tP_{4}$-saturated graphs. 
\begin{conjecture}\rm{(\cite{CFFGJM2015}, Conjecture 1.1)}\label{conjtpk}
Let $t\geq2$ be an integer. For $n$ sufficiently large, $$sat(n, tP_4)=\begin{cases}
		\frac{n+12t-12}{2} ,     &\text{if $n$ is even};  \\
		\frac{n+12t-11}{2} ,    & \text{if $n$ is odd}.\\
	\end{cases}$$\\
Moreover, $(t-1)N_4\cup\frac{1}{2}(n-12t+12)P_2\in \mathbf{SAT}(n, tP_4) $ if $n$ is even and $N^*_4\cup(t-2)N_4\cup\frac{1}{2}(n-12t+11)P_2\in \mathbf{SAT}(n, tP_4)$ if $n$ is odd, where $N_4$ and $N_4^*$ are shown in Figure \ref{N4N4s}.
\end{conjecture}

\begin{figure}[htbp]
 \hspace{0.08\textwidth}
	 \begin{minipage}{0.35\textwidth}
        \centering
        \includegraphics[width=0.8\linewidth]{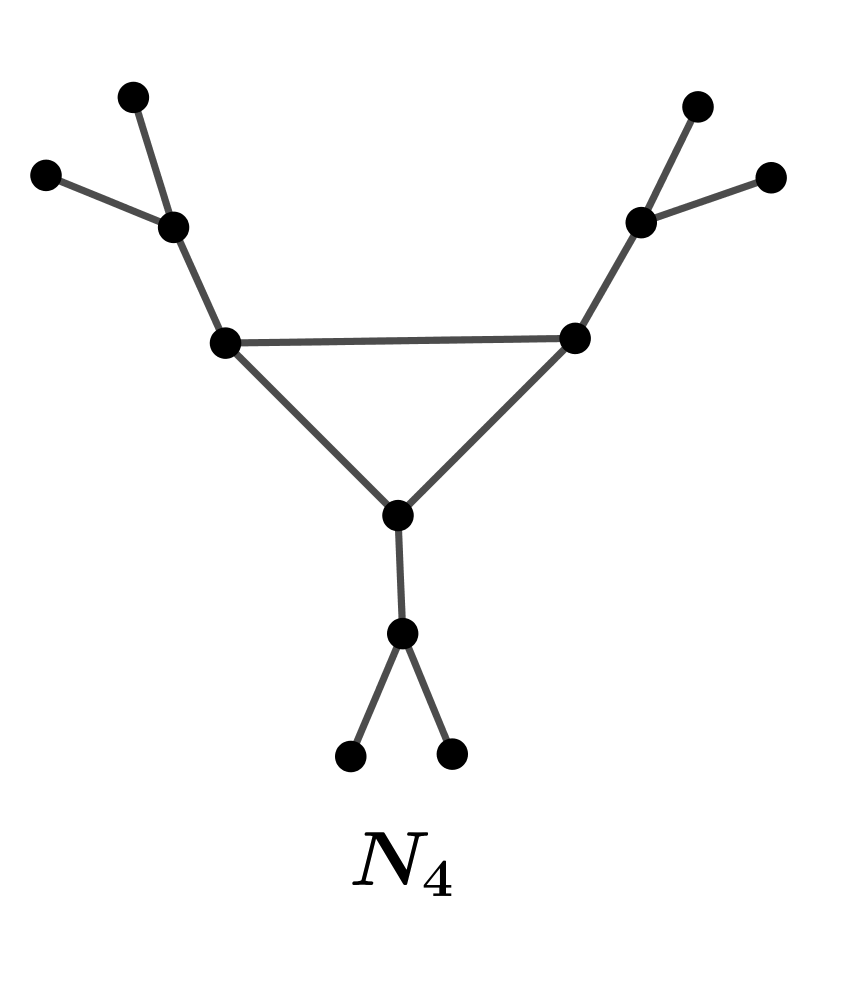} 
    \end{minipage}
    \hspace{0.1\textwidth} 
    \begin{minipage}{0.35\textwidth}
        \centering
        \includegraphics[width=0.8\linewidth]{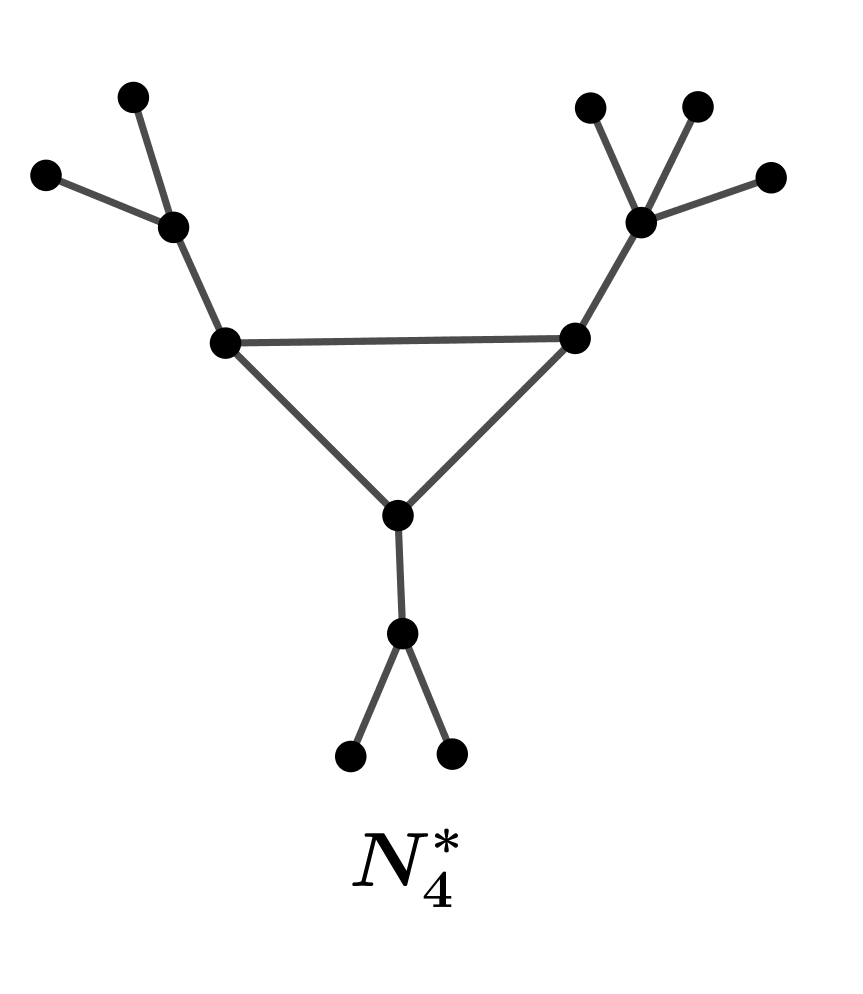} 
    \end{minipage}
     \caption{The graphs $N_4$ and $N_4^*$}
     \label{N4N4s}
\end{figure}

In the same paper, they also confirmed Conjecture \ref{conjtpk} when $t=2$ and $n\ge 12$ (see \cite{CFFGJM2015}, Theorem 4.4). Later,
The authors of \cite{CLYZ2023} proved that if $t\ge 3$ and $n\ge 6t+4$, then $sat(n,tP_4)<\frac{n+12t-12}{2} $, which shows that $(t-1)N_4\cup\frac{1}{2}(n-12t+12)P_2\notin \mathbf{SAT}(n, tP_4)$. In this paper, we determined the sharp lower bound on the spectral radius of $tP_4$-saturated graphs. The result demonstrates that $(t-1)N_4\cup\frac{1}{2}(n-12t+12)P_2\in \mathbf{SAT_{sp}}(n, tP_4)$ when $t\geq 2$,  and $\mathbf{SAT}(n, tP_4)\cap \mathbf{SAT_{sp}}(n, tP_4)=\emptyset$ when $t\ge 3$ and $n\ge 6t+4$. 

\begin{theorem}\label{tp4spec}
Let $t\ge 2$ and $G$ be a $tP_4$-saturated graph on $n$ vertices with $n\ge 4t$. Then $$\rho(G)\ge \rho(N_4)=\frac{1+\sqrt{17}}{2}$$ with equality if and only if $G\cong (t-1)N_4\cup Z$, where $Z\in \big\{\bigcup_{i=2}^{3}x_iK_i\cup (\bigcup_{i=4}^{7} x_iK_{1,i-1}): x_i\ge 0$ and $\sum_{i=2}^{7} i x_i=n-12t+12\big\}$. 
\end{theorem}

For $t=1$, observe that each component of $G$ is a star or a triangle. When $n$ is even, then $\rho(G)\ge 1$ with equality if and only if $G\cong \frac{n}{2} K_2$. When $n$ is odd, there exists an odd component $G_0$ such that $G_0\cong K_3$ or $K_{1,4}\subseteq G_0$. In this case, $\rho(G)\ge 2$ with equality if and only if $G\in \big\{\bigcup_{i\in [3]} x_iK_i\cup \bigcup_{j\in \{3,4\}} y_jK_{1,j}: \sum_{i\in [3]}ix_i+\sum_{j\in \{3,4\}}(j+1)y_j=n,  x_1\le 1,$  and $ x_1(x_2+ y_3+y_4)=0\big\}$. Hence, $\mathbf{SAT}(n, P_4)\subseteq \mathbf{SAT_{sp}}(n, P_4)$.

To prove Theorem \ref{tp4spec}, we derive a sufficient condition for graphs satisfying $\rho(G)\ge \rho(N_4)$. Let $F(v):= d^2(v)+\sum_{w\in N(v)\cup N_2(v)}\big(|N(w)\cap N(v)|\cdot d(w)\big)-\sum_{w\in N(v)}2d(w)-3d(v)+4$. 

\begin{theorem}\label{tp4suffcondi}
If $G$ is an $n$-vertex connected graph with $n\geq 3$ such that $ F(v)\ge0$ for each vertex $v \in V(G)$, then $\rho(G)\ge \rho(N_4)$ with equality if and only if $F(v)=0$ for each $v\in V(G)$. 
\end{theorem}

The paper is organized as follows: In Section \ref{tp3prop}, we introduce several notions and properties of $tP_{4}$-saturated graphs. Section \ref{toolpf} presents the proof of Theorem \ref{tp4suffcondi}. Building on this theorem, Section \ref{mainpf} provides the proof of Theorem \ref{tp4spec}.
 
\section{Preliminaries}\label{tp3prop}

 In this section, we introduce some definitions and known properties of $tP_4$-saturated graphs, and then establish some new ones. 
 
 For any $v\in V(G)$, use $d(v)$ and $N(v)$ to denote the degree and neighborhood of $v$ in $G$, respectively. Let $N[v]= N(v)\cup \{v\}$ and $N_2(v)=(\bigcup_{u\in N(v)}N(u))\setminus N[v]$. 
 For $A\subseteq V(G)$, denote by $G[A]$ the subgraph of $G$ induced by $A$.
 For any $e\in E(G)$ and $G_0\subseteq G$, we denote $G+e$ and $G-G_0$ the graphs obtained from $G$ by adding $e$ and deleting the vertices of $G_0$ together with their incident edges, respectively. Let $[k]=\{1,2,\ldots,k\}$.

\begin{lemma}\rm{(\cite{B2014}, Theorem 6.8)}\label{subgraphrho}
	Let $G$ be a connected graph and $H\subsetneq G$. Then $\rho(G)>\rho(H)$. 
\end{lemma}

An \textit{internal path} of $G$ is a path 
(or cycle) with vertices $v_1,v_2,\ldots, v_k$ (or $v_1=v_k$) such that $d(v_1)\geq 3$, $d(v_k)\geq 3$ and $d(v_2)=\cdots= d(v_{k-1})=2$.

\begin{lemma}\rm{(\cite{CRS2010}, Theorem 8.1.12)}\label{subdividingrho}
	Let $G$ be a connected graph with $n$ vertices, and let $G'$ be the graph obtained from $G$ by subdividing an edge on an internal path of $G$. Then $\rho(G')\leq \rho(G)$. 
\end{lemma}

\begin{lemma}\rm{(\cite{BEH2010}, Proposition 1.3.6)}\label{comspec}
 The spectrum of a graph is the union of the spectra of its connected components (and multiplicities are added). 
\end{lemma}

Denote by $s_{i}(A)$ the $i$-th row sum of a matrix $A$, i.e, $s_{i}(A)=\sum_{j=1}^n A_{ij}$.

\begin{lemma}\rm{(\cite{EZ2000}, Lemma 2.2)}\label{lemrho-p}
Let $G$ be an $n$-vertex connected graph and $A(G)=A$ its adjacency matrix, with spectral radius $\rho$. Let $p$ be any polynomial. Then
$$\min_{v\in V(G)}s_v(p(A))\leq p(\rho)\leq \max_{v\in V(G)}s_v(p(A)).$$
Moreover, if the rowsums of $p(A)$ are not all equal, then both inequalities are
strict.
\end{lemma}  
 
\begin{lemma}\rm{(\cite{CFFGJM2015}, Lemma 2.1 and Lemma 2.3)}\label{deg=1}
For any integer $t\geq 2$, let $G$ be a $tP_{4}$-saturated graph.  
\begin{enumerate}
  \item If $w\in V(G)$ with $d(w)=2$, then $w$ lies in a triangle.
 \item The graph $G$ is not a tree.
 \end{enumerate}
\end{lemma}

The following lemma provides some properties of $tP_{4}$-saturated graphs concerning vertices of degree at most $3$.

\begin{lemma}\label{lemdeg-eq}
    For any integer $t\geq 2$, let $G$ be a $tP_{4}${-}saturated connected graph. Let $u,w\in V(G)$ and $N(u)=\{u_1,\ldots,u_{d(u)}\}$, $N(w)=\{w_1,\ldots,w_{d(w)}\}$.
    \begin{enumerate}
        \item If $d(u)=d(w)=1$, $uw\notin E(G)$ and $u_1\ne w_1$, then $u_{1}w_{1}\notin E(G)$.  
        \item If $d(u)=2$ and $d(w)=3$, then $uw\notin E(G)$.  
    \end{enumerate}
\end{lemma}
\begin{proof} 
 For any two nonadjacent vertices in $G$, say $u$ and $v$, we have $G+uv$ contains a copy of $tP_4$, denoted $T(uv)$, and there is a copy of $P_4$ in $T(uv)$ containing $uv$, denoted $P_4(uv)$. 
 
    \begin{itemize}
  \item[(I)] Suppose $u_1w_1\in E(G)$. Since $G$ is $tP_{4}${-}saturated, we have $G+uw_1$ contains a copy of $tP_4$. 
  Note that $w\notin T(uw_1)- P_4(uw_1)$. 
  If $u_{1}\in P_4(uw_1)$, then $(T(uw_1)- P_4(uw_1))\cup \{u u_{1}w_{1}w\}$ is a copy of $tP_4$ in $G$, a contradiction. If $u_{1}\notin P_4(uw_1)$, then $u$ is one end of $P_4(uw_1)$ and $w\notin P_4(uw_1)$. Replacing $u$ with $w$ in $P_4(uw_1)$ results in a copy of $P_4$ in $G$, denoted by $P_4'$. In this case, $P_4'\cup (T(uw_1)- P_4(uw_1))$ is a copy of $tP_4$ in  $G$, a contradiction.

  \item[(II)] Suppose $uw\in E(G)$. Lemma \ref{deg=1} implies that $N(u)\cap N(w)\ne \emptyset$, say $N(u)\cap N(w)=\{v\}$. Let $w_1$ be the third neighbor of $w$. Then $uw_1\notin E(G)$ and $G+uw_1$ contains a copy of $tP_4$.  If $\{v,w\} \cap V(T(uw_1)-P_4(uw_1))=\emptyset$, then 
 replacing  $P_4(uw_1)$ with $\{vuww_1\}$ in $T(uw_1)$ results in a copy of $tP_4$ in $G$, a contradiction. Since $d(w)=3$, it is impossible for  $v\notin V(T(uw_1)-P_4(uw_1)) $  while $w\in V(T(uw_1)-P_4(uw_1))$. 
 Consider the case that  $v\in V(T(uw_1)-P_4(uw_1)) $ and $w\notin V(T(uw_1)-P_4(uw_1))$.  If $w\in P_4(uw_1)$, then $P_4(uw_1)=wuw_1w_2$, where $w_2\in N(w_1)$. Replacing $P_4(uw_1)$  with $uww_1w_2$ in $T(uw_1)$ results in a copy of $tP_4$ in $G$, a contradiction.  If $w\notin T(uw_1) $, then $u$ is  one end of $P_4(uw_1)$. Replacing $u$ with $w$ in $P_4(uw_1)$ results in a copy of $tP_4$ in $G$, a contradiction. 
Thus, assume that $v,w$ belong to $T(uw_1)-P_4(uw_1)$. Since $d(w)=3$, $v$ and $w$ belong to the same $P_4$, denoted $P_4'$. 
 Since $d(u)=2$ and $d(w)=3$, we see that $w$ is one end of $P_4'$ and $u$ is one end of $P_4(uw_1)$. Now, exchanging  the roles of $u$ and $w$ in $P_4(uw_1)$ and  $P_4'$ results in a copy of $tP_4$ in $G$, a contradiction.
 Hence, $uw\notin E(G)$.  
    \end{itemize}
    
 Therefore, the proof is complete. 
 \end{proof}

 Next, we present several properties of $tP_{4}$-saturated graphs concerning vertices of degree at least $3$. 
\begin{lemma}\label{lemdeggeq3}
    For any integer $t\geq 2$, let $G$ be a $tP_{4}${-}saturated connected graph. Let $u\in V(G)$ with $d(u)\geq 3$ and $N(u)=\{u_1,\ldots,u_{d(u)}\}$.
    \begin{enumerate}
        \item If there exist $u_1,u_2\in N(u)$ such that $d(u_1)=3$ \rm{(say $N(u_1)=\{u, u_{11}, u_{12}\}$)} and $d(u_2)=1$, then $u_{11}u_{12}\in E(G)$ and $d(u_{1i})\ge 3$ for any $i\in [2]$.
        \item If $u_1u_2\notin E(G)$ and $d(u_i)=1$ for any $u_i\in N(u)\setminus\{u_1,u_2\}$, then for $i\in [2]$, no two vertices $u_{i1},u_{i2} \in N(u_i)$ with degree $2$ are adjacent in $G$. 
        \item  If $E(G[N(u)])=\emptyset$ and there is a vertex $u_1\in N(u)$ such that $d(u_1)\geq 3$ and $d(u_{1i})=1$ for any $u_{1i}\in N(u_1)\setminus \{u\}$, then for each $u_j\in N(u)\setminus\{u_1\}$ with $d(u_j)\geq 3$, there are at least three vertices in $N(u_j)$ of degree greater than $1$.
        \item If $E(G[N(u)])=\{u_1u_2\}$ and $d(u_i)=1$ for any $u_i\in N(u)\setminus\{u_1,u_2\}$, then $d(u_i)\geq 3$ for any $i\in[2]$. 
    \end{enumerate}
\end{lemma}

 \begin{proof} 
 Similar to Lemma \ref{lemdeg-eq}, denote by $T(uv)$ a $tP_4$ in the graph $G+uv$ and denote by $P_4(uv)$ the $P_4$ in $T(uv)$ containing $uv$.
    \begin{itemize}
      \item[(I)] Suppose $u_{11}u_{12}\notin E(G)$. Then $G+u_{11}u_{12}$ contains a copy of $tP_4$. If $u\notin V(T(u_{11}u_{12})- P_4(u_{11}u_{12}))$, then 
      $u_1\notin V(T(u_{11}u_{12})- P_4(u_{11}u_{12})) $ and  $u_2\notin V(T(u_{11}u_{12})- P_4(u_{11}u_{12})) $.
      Replacing $P_4(u_{11}u_{12})$ with $u_2uu_1u_{11}$ in $T(u_{11}u_{12})$  results in a copy of $tP_4$ in $G$,  a contradiction. Assume that $u\in V(T(u_{11}u_{12})- P_4(u_{11}u_{12}))$.  In this case, we assert that there is a copy of $tP_4$ in $G+u_{11}u_{12}$, denoted $T(u_{11}u_{12})$, such that $u_1\notin V(T(u_{11}u_{12})- P_4(u_{11}u_{12}))$. 
      If $u_1\in V(T(u_{11}u_{12})- P_4(u_{11}u_{12}))$, by $d(u_1)=3$, then $u$ and $u_1$ lie on the same $P_4$ in $T(u_{11}u_{21})$, denoted  $P_4'$,  with $u_1$ as one end of $P_4'$. Replacing $u_1$ with $u_2$ in $P_4'$ results in a copy of $tP_4$ in $G+u_{11}u_{21}$, denoted $T'(u_{11}u_{21})$, such that $u_1\notin V(T'(u_{11}u_{12})- P_4(u_{11}u_{12}))$.
    Thus, we find the desired $tP_4$, as asserted.
      Note that $u_1\notin V(T(u_{11}u_{12})- P_4(u_{11}u_{12}))$ and $u\in  V(T(u_{11}u_{12})- P_4(u_{11}u_{12}))$. Then there exists a vertex $w\in N(u_{1i})\setminus \{u_1, u\}$ such that $w\in  P_4(u_{11}u_{21})$ for some $i\in [2]$. Replacing $P_4(u_{11}u_{12})$ with $u_{1j}u_1u_{1i}w$ for $\{i,j\}=[2]$ results in a copy of $tP_4$ in $G$, a contradiction. 
    Thus $u_{11}u_{12}\in E(G)$ and $d(u_{1i})\ge 2$, where $i\in [2]$. Since $d(u_1)=3$, 
    Lemma \ref{lemdeg-eq} (II) implies that $d(u_{1i})\ge 3$ for any $i\in [2]$.  

\item[(II)] Suppose that there are two vertices $u_{i1}, u_{i2} \in N(u_i)$ with $d(u_{i1})=d(u_{i2})=2$ and $u_{i1}u_{i2}\in E(G)$ for some $i\in [2]$. Then $G+u_1u_2$ contains a copy of $tP_4$ as $u_1u_2\notin E(G)$. In this case $u,u_{i1}, u_{i2} \notin V(T(u_1u_2)-P_4(u_1u_2))$. However, $(T(u_1u_2)-P_4(u_1u_2)) \cup \{uu_iu_{i1}u_{i2}\}$ is a copy of $tP_4$ in $G$,  a contradiction. 

\item[(III)] Let $N(u_i)=\{u,u_{i1},\ldots,u_{i(d(u_i)-1)}\}$. Suppose $d(u_2)\ge 3$ and there are at most two vertices in $N(u_2)$ of degree greater than $1$, that is, $d(u_{2j})=1$ for any $u_{2j}\in N(u_2)\setminus\{u,u_{21}\}$. Since $uu_{21}\notin E(G)$, $G+uu_{21}$ contains a copy of $tP_4$. In this case $u_{1}, u_{2} \notin V(T(uu_{21})-P_4(uu_{21}))$. However, $(T(uu_{21})-P_4(uu_{21})) \cup \{u_1uu_{2}u_{21}\}$ is a copy of $tP_4$ in $G$, a contradiction. 

\item[(IV)] Suppose that $d(u_1)=2$. Since $u_2u_3\notin E(G)$ where $u_3\in N(u)$, we see that $G+u_2u_3$ contains a copy of $tP_4$ and $u, u_{1} \notin V(T(u_2u_3)-P_4(u_2u_3))$. It follows that $(T(u_2u_3)-P_4(u_2u_3)) \cup \{u_2u_1uu_3\}$ is a copy of $tP_4$ in $G$, which is a contradiction. 
    \end{itemize}
    
     Therefore, the proof is complete. 
 \end{proof}
                                                 
\section{The Proof of Theorem \ref{tp4suffcondi}}\label{toolpf}
\begin{theproof}[Proof of Theorem \ref{tp4suffcondi}]
Let $\rho$ be the spectral radius of $G$. Suppose that $p(x)=x^3-2x^2-3x+4.$ By Lemma \ref{lemrho-p}, we have 
	\begin{align}\label{fA}
		p(\rho) \ge&   \min_{v\in V(G)} \bigg\{\sum_{w\in V(G)}(A^3-2A^2-3A+4I)_{vw}\bigg\}.
	\end{align}

\begin{claim}\label{A3}
For any fixed $i\in V(G)$, we have 	$$\sum_{j\in V(G)}(A^3)_{ij}=d^2(i)+\sum_{k\in N(i)\cup N_2(i)}|N(k)\cap N(i)|\cdot d(k).$$ 
\end{claim}
\begin{proof}
Note that $\sum_{j\in V(G)}(A^3)_{ij}$ is equal to the number of all walks of length three starting from $i$ in $G$. 
Let $S=\{i i_1 i_2 i_3: $ where $i i_1 i_2 i_3$ is a walk of length three starting from $i$ in $G$\}, $S_1=\{i i_1 i_2 i_3\in S: i_2=i\}$ and $S_2=\{i i_1 i_2 i_3\in S: i_2\ne i\}$. 
Then $|S|=\sum_{j\in V(G)}(A^3)_{ij}$.
Note that $|S_1|=d^2(i)$ and 
\begin{align*}
	|S_2|&=\sum_{i_1\in N(i)}\sum_{i_2\in N(i_1)\setminus\{i\}}d(i_2)
	= \sum_{k\in N(i)\cup N_2(i)}|N(k)\cap N(i)| \cdot d(k).
\end{align*}
Thus $$\sum_{j\in V(G)}(A^3)_{ij}=|S|=|S_1|+|S_2|=d^2(i)+\sum_{k\in N(i)\cup N_2(i)}|N(k)\cap N(i)| \cdot d(k).$$
Hence, Claim \ref{A3} follows.
 \end{proof}
Combining Inequality (\ref{fA}) with Claim \ref{A3} yields that 
	\begin{align*}
	p(\rho)\ge& \min_{v\in V(G)}  \bigg\{ d^2(v)+\sum_{w\in N(v)\cup N_2(v)}|N(w)\cap N(v)| \cdot d(w)-\sum_{w\in N(v)}2d(w)
	-3d(v)+4 \bigg\}.
	\end{align*}
    By assumptions of Theorem \ref{tp4suffcondi}, we get that $p(\rho)\ge \min_{v\in V(G)} \{F(v)\}\ge 0$. 
Denote the three roots of $p(x)=0$ by $r_1,r_2,r_3$. Let $r_1\ge r_2\ge r_3$. We can see $r_1\cdot r_2\cdot r_3=-4<0$ and so $r_3<0$. 
By calculation, we find $\rho(N_4)=r_1=\frac{1+\sqrt{17}}{2}$ and $r_2=1$. From $p(\rho)\ge 0$, it follows that either $\rho\ge r_1$ or $r_3\le \rho \le r_2=1$. Observe that $G$ contains a copy of $P_3$ since $n\ge 3$. It implies that $\rho >1$ and so $\rho \ge \rho(N_4)$ with equality if and only if $F(v)=0$ for each $v\in V(G)$. This completes the proof. 
\end{theproof}

\section{The Proof of Theorem \ref{tp4spec}}\label{mainpf}
Let $e(A,B)$ denote the number of edges in $G$ with one end in $A$ and the other in $B$. In addition, we simply write $e(A,A)$ as $e(A)$. For a vertex $v\in V(G)$, define $A=N(v)$, $B=N_2(v)$ and $C=V(G)\setminus(A\cup B\cup\{v\})$. 
We have the following fact. 
\begin{fact}\label{fact1}
  If $d(w)\geq 2$ for any $w\in N_2(v)$, then $$\sum_{w\in N_2(v)}|N(w)\cap N(v)| \cdot d(w)-2e(A,B)\geq 0,$$ where equality holds if and only if $d(w)=2$ for any $w\in N_2(v)$.    
\end{fact}

\begin{proof}
 Let $B_1=\{w: w\in N_2(v)\text{ and }|N(w)\cap N(v)|=1\}$ and $B_2=\{w: w\in N_2(v)\text{ and }|N(w)\cap N(v)|\geq 2\}$. Then
    \begin{align*}
        \sum_{w\in N_2(v)}|N(w)\cap N(v)| \cdot d(w)&\geq \sum_{w\in B_1}d(w)+\sum_{w\in B_2}2d(w)\\
&\geq 2|B_1|+2e(A,B_2)\\
&= 2e(A,B_1)+2e(A,B_2)\\
&= 2e(A,B),
    \end{align*}
     and equality holds if and only if $d(w)=2$ for any $w\in N_2(v)$. 
\end{proof}

Note that the graphs $F_1$, \ldots, $F_{20}$ referred to in the subsequent proofs are illustrated in Figure \ref{graphsF}. 
\begin{lemma}\label{deg3}
   Let $t\geq 2$ and $G$ be a connected $tP_4$-saturated graph on $n$ vertices with $n\ge 4t$ and $\rho(G)\le \rho(N_4)$. For $v \in V(G)$ with $d(v)=3$ and $e(A)=0$, let $N(v)=\{v_1, v_2, v_3\}$.  Then \begin{align}\label{d3A0goal}
    \sum_{w\in   N_2(v)} |N(w)\cap N(v)|\cdot d(w) \ge  2e(A,B)+2,
\end{align}
where equality holds if and only if  $N(v)$ satisfies: both $v_1$ and $v_2$ have degree $1$, while $v_3$ has degree 
 $3$ with $N(v_3)=\{v,v_{31},v_{32}\}$, $d(v_{31})=d(v_{32})=3$ and $v_{31}v_{32}\in E(G)$. 
\end{lemma}
   \begin{proof} 
    We first conclude that if there is a vertex $v_1\in N(v)$ such that $d(v_1)=1$, then Inequality \eqref{d3A0goal} holds.
Since $G$ is connected and $n\ge 4t\ge 8$, we have $ d(v_2)\ge 2$ or $d(v_3)\ge 2$. Assume that $d(v_3)\ge 2 $. By $e(A)=0$, Lemma \ref{deg=1} (I) implies that $d(v_3)\ge 3$. 
Since $d(v_1)=1$, Lemma \ref{lemdeg-eq} (I) yields that there is no degree $1$ vertex in $N(v_2)\cup N(v_3)$, which means that for any vertex $u\in N_2(v)$, we have $d(u)\ge 2$. By Fact \ref{fact1}, we obtain $\sum_{w\in   N_2(v)} |N(w)\cap N(v)|\cdot d(w) \ge  2e(A,B)$. 
Since $d(v_3)\ge 3 $, there is a vertex $v_{31}\in N(v_3)\cap N_2(v)$ such that $d(v_{31})\ge 2$. If $d(v_{31})=2$, by Lemma \ref{deg=1} (I), there is another vertex $v_{32}\in N(v_3)\cap N_2(v)$ such that $v_{31}v_{32}\in E(G)$. Lemma \ref{lemdeg-eq} (II)  and Lemma \ref{lemdeggeq3} (II) imply that $d(v_3)\ge 4$ and $d(v_{32})\ge 4$. Assume that the fourth neighbor of $v_3$ is $v_{33}$. If $d(v_{33})=2$,  by Lemma \ref{deg=1} (I), then $F_1\subseteq G $ or $F_2\subseteq G$. In both cases, $\rho(G)> \rho(N_4)$, a contradiction. This implies $d(v_{33})\geq 3$ and consequently $\sum_{w\in   N_2(v)} |N(w)\cap N(v)|\cdot d(w) >  2e(A,B)+2$. Assume that any vertex in $N(v_3)\cap N_2(v)$ has degree at least $3$. Since $d(v_3)\ge 3$, there are at least two vertices in $N(v_3)\cap N_2(v)$ have degree at least $3$. Hence, Inequality \eqref{d3A0goal} holds, and equality holds only when $d(v_1)=d(v_2)=1$, $d(v_3)=3$, and $d(v_{31})=d(v_{32})=3$, where $N(v_3)=\{v, v_{31}, v_{32}\}$. Furthermore, Lemma \ref{lemdeggeq3} (I) implies that $v_{31}v_{32}\in E(G)$. 
 
Now, assume that $d(v_i)\ge 2$ for any $i\in [3]$. In fact, by Lemma \ref{deg=1} (I) and $e(A)=0$, we have $d(v_i)\ge 3. $
If there is a vertex $w_1\in N_2(v)$ such that $d(w_1)=1$, say $N(w_1)=v_1$, by Lemma \ref{lemdeggeq3} (I), then $v_2v_3\in E(G)$, which contradicts $e(A)=0$. 
It remains to consider the case where for any $w\in N_2(v)$, $d(w)\ge 2$. If there is a vertex $v_{i1}\in N(v_i)\cap N_2(v)$ such that $d(v_{i1})\ge 3$ for any $i\in [3]$, then $\sum_{w\in   N_2(v)} |N(w)\cap N(v)|\cdot d(w) >  2e(A,B)+2$. Thus, assume that there exists a vertex, say $v_1$, such that any vertex in $N(v_1)\cap N_2(v)$ has degree $2$. Lemma \ref{lemdeg-eq} (II) implies that $d(v_1)\ge 4$. By Lemma \ref{deg=1} (I), there is a copy of $F_1$ or $F_2$ in $G[N[v_1]]$. It follows that $\rho(G)>\rho(N_4)$ since $\rho(F_1)=\rho(F_2)=\rho(N_4)$, which contradicts $\rho(G)\le \rho(N_4)$. 
   \end{proof}
   
 \begin{lemma}\label{contp4}
Let $t\ge 2$ and $G$ be a connected $tP_4$-saturated graph on $n$ vertices with $n\ge 4t$ and $\rho(G)\le \rho(N_4)$. Then $G\cong N_4$. 
\end{lemma}

\begin{proof}
We first prove $\rho(G)\ge \rho(N_4)$. 
By Theorem \ref{tp4suffcondi}, it suffices to show that $F(v)\ge 0$ for each $v\in V(G)$. 
By calculations, we have 
\begin{align}
 F(v)
=&d^2(v)+\sum_{w\in N(v)}|N(w)\cap N(v)| \cdot (d(w)-2)+2\sum_{w\in N(v)}|N(w)\cap N(v)|\notag\\
&+\sum_{w\in N_2(v)}|N(w)\cap N(v)| \cdot d(w)- 2d(v)-4e(A)-2e(A,B) -3d(v)+4\notag\\
=&d^2(v)-5d(v)+4+\sum_{w\in N(v)}|N(w)\cap N(v)| \cdot (d(w)-2)+2\cdot2e(A)\notag\\
&+\sum_{w\in N_2(v)}|N(w)\cap N(v)| \cdot d(w)-4e(A)-2e(A,B)\notag\\
=&d^2(v)-5d(v)+4+\sum_{w\in N(v)}|N(w)\cap N(v)| \cdot (d(w)-2) \notag\\
&+\sum_{w\in N_2(v)}|N(w)\cap N(v)| \cdot d(w)-2e(A,B).\label{eqkey}
\end{align}
 
 Let $N(v)=\{v_1,\ldots,v_{d(v)}\}$ and $N(v_i)=\{v,v_{i1},v_{i2},\ldots,v_{i(d(v_i)-1)}\}$ for $1\leq i\leq d(v)$. Since $\rho(G)\le \rho(N_4)=\frac{1+\sqrt{17}}{2}$, we have $\Delta(G)\le 6$. 
\begin{case}
    $d(v)=1$.
\end{case}
In this case, Equation \eqref{eqkey} becomes
$$F(v)=\sum_{w\in N_2(v)}d(w)-2e(A,B)=2e(B)+e(B,C)-e(A,B).$$ Note that $e(B)\leq 1$, otherwise $\rho(G)>\rho(N_4)$ since $F_1$ or $F_2$ is a subgraph of $G$. If $e(B)= 1$, then $3\leq d(v_1)\leq 5$ by $ \rho(F_3)>\rho(N_4)$. From Lemma \ref{lemdeggeq3} (IV) and Lemma \ref{deg=1} (I), $e(B,C)\geq 2$. Hence,
\begin{align*}
    F(v)=2+e(B,C)-e(A,B)\geq 4-e(A,B)=5-d(v_1)\geq 0.
\end{align*}

If $e(B)=0$, then $3\leq d(v_1)\leq 6$ and $d(v_{1i})=1$ or $d(v_{1i})\geq 3$ for $1\leq i\leq d(v_1)-1$ by Lemma \ref{deg=1} (I). 
\begin{itemize}
    \item [\rm i)]  $d(v_1)=6$: Note that $d(v_{1i})\leq 3$ for $i\in [5]$ by $\rho(F_4)>\rho(N_4)$. And all vertices in $N(v_1)$ except one are of degree $1$ since $\rho(F_5)>\rho(N_4)$ and $n\geq8$. Without loss of generality, assume that $d(v_{15})=3$ and $d(v_{1i})=1$ for $i\in [4]$. Let $N(v_{15})=\{v_1,v_{151},v_{152}\}$. Then $v_{151}v_{152}\in E(G)$ by Lemma \ref{lemdeggeq3} (I). It implies $\rho(G)\geq  \rho(F_6)>\rho(N_4)$, a contradiction.
    
    \item [\rm ii)]  $d(v_1)=5$: When $e(B,C)\geq e(A,B)=4$, then $F(v)\geq 0$. When $e(B,C)\leq 3$, by Lemma \ref{deg=1} (I) and $n\geq 8$, all vertices in $N(v_1)$ have degree $1$ except for one vertex, say $v_{11}$ with $3\leq d(v_{11})\leq 4$. If $d(v_{11})=4$, then $E(G[N(v_{11})])=\emptyset$ otherwise $\rho(G)\geq \rho(F_7)>\rho(N_4)$. With Lemma \ref{deg=1} (I) and Lemma \ref{lemdeg-eq} (I), we have $d(v_{11j})\geq 3$ for any $v_{11j}\in N(v_{11})$. According to Lemma \ref{subgraphrho} and Lemma \ref{subdividingrho}, we see that $\rho(G)\geq  \rho(F_8)>\rho(N_4)$, a contradiction. If $d(v_{11})=3$, let $N(v_{11})=\{v_1,v_{111},v_{112}\}$. Then $v_{111}v_{112}\in E(G)$ and $d(v_{11j})\geq3$ for $i\in [2]$ by Lemma \ref{lemdeggeq3} (I), which implies that $\rho(G)\geq \rho(F_9)>\rho(N_4)$, a contradiction.
    
    \item [\rm iii)]   $d(v_1)=4$: When $e(B,C)\geq e(A,B)=3$, then $F(v)\geq 0$. When $e(B,C)\leq 2$, assume that $d(v_{11})=3$, $d(v_{12})=1$ and $d(v_{13})=1$ according to $n\geq 8$ and Lemma \ref{deg=1} (I). Let $N(v_{11})=\{v_1,v_{111},v_{112}\}$. Then $v_{111}v_{112}\in E(G)$ and $d(v_{11j})\geq3$ for $i\in [2]$ by Lemma \ref{lemdeggeq3} (I). Moreover, $|N(v_{111})\cap N(v_{112})|\leq1$ and $e(N(v_{11i}))=1$ for $i\in [2]$ since $\rho(F_1)=\rho(F_2)=\rho(N_4)$. From Lemma \ref{deg=1} (I), Lemma \ref{lemdeg-eq} (I) and Lemma \ref{subdividingrho}, it follows that $\rho(G)\geq \rho(F_{10})>\rho(N_4)$, a contradiction.
    
    \item [\rm iv)] $d(v_1)=3$: By Lemma \ref{deg=1} and Lemma \ref{lemdeg-eq} (II) and $n\geq 8$, we have $e(B,C)\geq 2=e(A,B)$. Hence $F(v)\geq 0$.
\end{itemize} 

Hence, for $v\in V(G)$ with $d(v)=1$, we have $F(v)\geq 0$ with equality if and only if $e(B)=1$, $d(v_1)=5$ and $e(B,C)=2$; or $e(B)=0$ and $2\leq e(B,C)=d(v_1)-1\leq 4$.
\begin{case}
    $d(v)=2$.
\end{case}
In this case $v_1v_2\in E(G)$ by Lemma \ref{deg=1} (I) and 
$$F(v)=-2+\sum_{w\in N_2(v)}|N(w)\cap N(v)| \cdot d(w)-e(A,B).$$
Observe that $|N(w)\cap N(v)|=1$ for any $w\in B$ otherwise $\rho(G)>\rho(F_1)=\rho(N_4)$. Then \begin{align*}
    F(v)&= -2+\sum_{w\in N_2(v)}d(w)-e(A,B)\\
&=-2+e(A,B)+2e(B)+e(B,C)-e(A,B)\\
     &=2e(B)+e(B,C)-2.
\end{align*}
If $e(B)=0$ and $e(B,C)\leq 1$, then $e(B,C)=0$ by Lemma \ref{deg=1} (I). It follows from $ n\geq 8$ that $F_{11}$ or $F_{12}$ is a subgraph of $G$, which contradicts $\rho(G)\leq \rho(N_4)$. If $e(B)\geq 1$ or $e(B,C)\geq 2$, then $F(v)\geq 0$ and equality holds only if either $e(B)=1$ and $e(B,C)=0$, or $e(B)=0$ and $e(B,C)=2$. Note that when $e(B)=1$ and $e(B,C)=0$, by Lemma \ref{deg=1} (I) and $|N(w)\cap N(v)|=1$ for any $w\in B$, we have $ \rho(G)>\rho(F_2)=\rho(N_4)$, a contradiction. It follows that $F(v)\geq 0$ for $v\in V(G)$ with $d(v)=2$, with equality if and only if $e(A)=1$ and there is exactly one vertex in $B$ with degree $3$ while the others (if exist) have degree $1$.

\begin{case}
    $d(v)=3$.
\end{case}

If $e(A)\ge 2$, then  $\rho(G)>\rho(F_1)=\rho(N_4)$ by $n\ge 8$. We only need to consider the two cases $e(A)=1$ and $e(A)=0$. 

If $e(A)=1$, say $E(G[A])=\{v_1v_2\}$, then $$F(v)= d(v_1)+d(v_2)-6-2e(A,B)+\sum_{w\in N_2(v)} |N(w)\cap N(v)|\cdot d(w).$$
When $d(v_3)=1$, then $e(A,B)=d(v_1)+d(v_2)-4$ and
$d(w)\geq 3$ for any $w\in B$ by Lemma \ref{lemdeg-eq} (I) and $\rho(F_1)=\rho(F_2)=\rho(N_4)$. Thus, according to Fact \ref{fact1} and Lemma \ref{lemdeggeq3} (IV), we have 
 \begin{align*}
    F(v)&=-2-e(A,B)+\sum_{w\in N_2(v)} |N(w)\cap N(v)|\cdot d(w)\\
    &> -2-e(A,B)+2e(A,B)\\
    &=e(A,B)-2\geq 0.
    \end{align*} 
Assume $d(v_3)\geq 2$. By Lemma \ref{deg=1} (I) and Lemma \ref{lemdeg-eq} (II), we get $d(v_i)\geq 3$ for $i\in [3]$. Then $d(v_1)=d(v_2)=3$ otherwise $\rho(G)\geq \rho(F_{13})>\rho(N_4)$ by Lemma \ref{subgraphrho} and Lemma \ref{subdividingrho}. Let $N(v_i)\setminus\{v,v_1,v_2\}=\{v_{i1}\}$ for $i\in [2]$. We claim that $d(v_{i1})\geq 3$ for any $i\in [2]$. Otherwise, according to Lemma \ref{lemdeg-eq} (II), assume $d(v_{11})=1$. Since $G+vv_{21}$ contains a copy of $tP_4$, we find that $(T(vv_{21})-P_4(vv_{21})) \cup \{vv_1v_2v_{21}\}$ is a copy of $tP_4$ in $G$, a contradiction. So $d(v_3)=3$ by $\rho(F_{14})>\rho(N_4)$ and Lemma \ref{subdividingrho}. Thus,
\begin{align*}
   F(v)\geq -8+d(v_{11})+d(v_{21})+d(v_3)-1\geq 0.
    \end{align*} 

 If $e(A)=0$, then $|N(w)\cap N(v)|=0$ for any $w\in N(v)$ and by Lemma \ref{deg3},
 \begin{align*}
 F(v)=-2-2e(A,B)+\sum_{w\in   N_2(v)} |N(w)\cap N(v)|\cdot d(w)\geq 0.
    \end{align*}
    
 Hence, for $v\in V(G)$ with $d(v)=3$, we have $F(v)\geq 0$ with equality if and only if one of the following holds: (i) $E(A)=\{v_1v_2\}$, $d(v_i)=d(v_{11})=d(v_{21})=3$ for $i\in [3]$ and $d(v_{31})=d(v_{32})=1$; (ii) $e(A)=0$, $d(v_1)=d(v_2)=1$ and 
 $d(v_3)=3$ with $N(v_3)=\{v,v_{31},v_{32}\}$,  $d(v_{31})=d(v_{32})=3$ and $v_{31}v_{32}\in E(G)$.
 
\begin{case}
    $d(v)\geq 4$.
\end{case}
 Note that $e(A)\leq 1$ since $\rho(F_1)=\rho(F_2)=\rho(N_4)$ and 
 $$F(v)\geq \sum_{w\in N_2(v)}|N(w)\cap N(v)| \cdot d(w)-2e(A,B).$$
 If there is a vertex in $A$ of degree one, then $d(w)\geq 2$ for each $w\in B$ by Lemma \ref{lemdeg-eq} (I). It follows from Fact \ref{fact1}, that $F(v)\geq 0$. Assume $d(u)\geq 2$ for any $u\in A$. Then $d(v)=4$. Otherwise according to Lemma \ref{deg=1} (I) and Lemma \ref{subdividingrho}, we have $\rho(G)\geq \rho(F_{19})>\rho(N_4)$ or $\rho(G)\geq \rho(F_{20})>\rho(N_4)$, a contradiction. 
 
 When $e(A)=1$ (say $E(G[A])=\{v_1v_2\}$), then $d(v_3)\geq 3$ and $d(v_4)\geq 3$ according to Lemma \ref{deg=1} (I). Moreover, we claim that $d(v_1)\geq 3$ or $d(v_2)\geq 3$. Otherwise, $d(v_1)=d(v_2)=2$. Since $G+v_3v_4$ contains a copy of $tP_4$, we find that $(T(v_3v_4)-P_4(v_3v_4)) \cup \{v_1v_2vv_3\}$ is a copy of $tP_4$ in $G$, a contradiction. With Lemma \ref{subdividingrho}, we obtain $\rho(G)\geq \rho(F_{15})>\rho(N_4)$, a contradiction.

 When $e(A)=0$, then $d(v_i)\geq 3$ for $i\in [4]$ by Lemma \ref{deg=1} (I). Firstly, we assert that $d(v_i)\leq 4$ for $i\in [4]$ and there is at most one vertex of degree $4$ in $A$, otherwise $\rho(G)>\rho(F_{16})>\rho(N_4)$ or $\rho(G)>\rho(F_{17})=\rho(N_4)$ by Lemma \ref{deg=1} (II) and Lemma \ref{subdividingrho}, a contradiction. Consider $d(v_i)=3$ for $i\in [4]$. Recall that $N(v_i)=\{v,v_{i1},v_{i2},\ldots,v_{i(d(v_1)-1)}\}$ for $1\leq i\leq d(v)$. If there is a vertex in $N(v)$, say $v_4$, such that $d(w)=1$ for each $w\in N(v_4)\setminus\{v\}$, then $d(v_{ij})\geq 3$ for $i\in [3]$ and $j\in [2]$ by Lemma \ref{lemdeggeq3} (III) and Lemma \ref{lemdeg-eq}  (II), which implies that $F(v) =\sum_{i\in [4],j\in [2]}d(v_{ij})-2e(A,B)\geq6\times 3+2\times1-2\times 8>0.$
So assume that $d(v_{i1})\geq 2$ for each $i\in[4]$. With Lemma \ref{lemdeg-eq} (II), we get $d(v_{i1})\geq 3$ where $i\in[4]$. Then $ F(v)=\sum_{i\in [4],j\in [2]}d(v_{ij})-2e(A,B)
    \geq4\times1+4\times 3-2\times 8=0.$
Now, suppose that $d(v_i)=3$ for $i\in [3]$ and $d(v_4)=4$. If there is a vertex $v_i\in N(v)$ such that $d(w)=1$ for each $w\in N(v_i)\setminus\{v\}$, then we derive $F(v)>0$ by a similar reasoning to above. Assume $d(v_{i1})\geq 2$ for each $i\in[4]$. Note that $d(v_{i1})\geq 3$ for each $i\in[3]$ by Lemma \ref{lemdeg-eq} (II). Then $\rho(G)> \rho(F_{18})>\rho(N_4)$ from Lemma \ref{subdividingrho}, a contradiction. 

Thus, for $v\in V(G)$ with $d(v)\geq 4$, we have $F(v)\geq 0$ with equality if and only if $d(v)=4$ and one of the following holds: (i) $d(v_1)=1$, $d(v_2)=d(v_3)=2$, $E(A)=\{v_2 v_3\}$, and $d(w)=2$ for any $w\in N_2(v)$;  (ii) $d(v_1)=1$, $e(A)=0$, and $d(w)=2$ for any $w\in N_2(v)$; (iii) $e(A)=0$, $d(v_i)=d(v_{i1})=3$, and $d(v_{i2})=1$ for $i\in [4]$.

Consequently, we obtain $F(v)\geq 0$ for each $v\in V(G)$. By Theorem \ref{tp4suffcondi} and $\rho(G)\leq \rho(N_4)$, we conclude that $\rho(G)=\rho(N_4)$ and $F(v)= 0$ for each $v\in V(G)$. By checking the conditions for $F(v)= 0$, we find that $d(v)\leq 4$. The conditions for $F(v)= 0$ when $d(v)=2$ and $d(v)=3$ imply that $d(v)\neq 2$ for any $v\in V(G)$. It follows that if $d(v)=1$, then $3\leq d(v_1)\leq 4$ and all the vertices in $N_2(v)$ have degree $1$, except for one vertex with degree $d(v_1)$. Combining with the above conditions for $F(v)= 0$, we deduce that $d(v)\neq 4$ for any $v\in V(G)$. Therefore, it follows that $G\cong N_4$.
\end{proof}

\begin{theproof}[Proof of Theorem \ref{tp4spec}] 
Let $G$ be a $tP_4$-saturated graph with $t\ge 2$ and $n\ge 4t$. Then there is a copy of $(t-1)P_4$ in $G$ and there is a component of $G$ contains a copy of $P_4$, denoted by $G_1$. 
Clearly, $G_1$ is also $t_1P_4$-saturated for some $t_1\ge 2$. If $|V(G_1)|<4t_1$, then $G_1$ must be a clique. Then $\rho(G)\geq \rho(G_1)\ge \rho(K_4)=3>\rho(N_4)$. We may assume that $|V(G_1)|\ge 4t_1$. By Lemma \ref{contp4}, $\rho(G)\ge \rho(G_1)\ge \rho(N_4)$. 
The equality holds if and only if each component of $G$ containing $P_4$ is isomorphic to $N_4$. Since $G$ is $tP_4$-saturated, there are $t-1$ copies of $N_4$ in $G$ and no isolated vertices. Moreover, the remaining components are $P_4$-saturated and each has spectral radius no more than $\rho(N_4)$. It follows that $G\cong (t-1)N_4\cup Z$ where $Z\in \big\{\bigcup_{i=2}^{3}x_iK_i\cup (\bigcup_{i=4}^{7} x_iK_{1,i-1}): x_i\ge 0$ and $\sum_{i=2}^{7} i x_i=n-12t+12\big\}$.
\end{theproof}

 \begin{figure}[H]
    \centering 
\subfigure[$F_1: \rho=\frac{1+\sqrt{17}}{2}$]{
    \begin{minipage}{0.22\textwidth}
      \centering
  	\begin{tikzpicture}[scale=0.4]
\node[circle,fill=black,draw,inner sep=0pt,minimum size=1mm] (x) at (2.5,2) {};
            \node[circle,fill=black,draw,inner sep=0pt,minimum size=1mm] (a) at (0,0) {};
            \node[circle,fill=black,draw,inner sep=0pt,minimum size=1mm] (b) at (2.5,0){};
            \node[circle,fill=black,draw,inner sep=0pt,minimum size=1mm] (c) at (5,0){};
            \draw (x)--(a) -- (b)--(c) -- (x);
            \draw(b) -- (x); 
	\end{tikzpicture}  
	 \end{minipage}
}
\subfigure[$F_2: \rho=\frac{1+\sqrt{17}}{2}$]{
    \begin{minipage}{0.22\textwidth}
      \centering
	\begin{tikzpicture}[scale=0.4]
\node[circle,fill=black,draw,inner sep=0pt,minimum size=1mm] (x) at (2.5,2) {};
            \node[circle,fill=black,draw,inner sep=0pt,minimum size=1mm] (a) at (0,0) {};
            \node[circle,fill=black,draw,inner sep=0pt,minimum size=1mm] (b) at (1.5,0){};
            \node[circle,fill=black,draw,inner sep=0pt,minimum size=1mm] (c) at (3.5,0){};
             \node[circle,fill=black,draw,inner sep=0pt,minimum size=1mm] (d) at (5,0){};
            \draw (x)--(a) -- (b)--(x);
            \draw (x)--(c) -- (d)--(x); 
	\end{tikzpicture} 
\end{minipage}
}
\subfigure[$F_3: \rho=2.68133$]{
    \begin{minipage}{0.22\textwidth}
      \centering
  	\begin{tikzpicture}[scale=0.4]
 \node[circle,fill=black,draw,inner sep=0pt,minimum size=1mm] (x) at (2.5,2) {};
            \node[circle,fill=black,draw,inner sep=0pt,minimum size=1mm] (a) at (0,0) {};
            \node[circle,fill=black,draw,inner sep=0pt,minimum size=1mm] (b) at (1,0){};
            \node[circle,fill=black,draw,inner sep=0pt,minimum size=1mm] (c) at (2,0){};
            \node[circle,fill=black,draw,inner sep=0pt,minimum size=1mm] (d) at (3,0){};
            \node[circle,fill=black,draw,inner sep=0pt,minimum size=1mm] (e) at (4,0){};
            \node[circle,fill=black,draw,inner sep=0pt,minimum size=1mm] (f) at (5,0){};
            \draw (x)--(a) -- (b)--(x);
            \draw (x)--(c);
            \draw (x)--(d);
            \draw (x)--(e);
             \draw (x)--(f);
	\end{tikzpicture}  
	 \end{minipage}
    } 
	 \subfigure[$F_4: \rho=2.60601$]{
    \begin{minipage}{0.22\textwidth}
      \centering
  	\begin{tikzpicture}[scale=0.4]
 \node[circle,fill=black,draw,inner sep=0pt,minimum size=1mm] (x) at (2.5,2) {};
            \node[circle,fill=black,draw,inner sep=0pt,minimum size=1mm] (a) at (0,1) {};
            \node[circle,fill=black,draw,inner sep=0pt,minimum size=1mm] (b) at (1,1){};
            \node[circle,fill=black,draw,inner sep=0pt,minimum size=1mm] (c) at (2,1){};
            \node[circle,fill=black,draw,inner sep=0pt,minimum size=1mm] (d) at (3,1){};
            \node[circle,fill=black,draw,inner sep=0pt,minimum size=1mm] (e) at (4,1){};
            \node[circle,fill=black,draw,inner sep=0pt,minimum size=1mm] (f) at (5,1){};
              \node[circle,fill=black,draw,inner sep=0pt,minimum size=1mm] (a1) at (-1,0){};
                \node[circle,fill=black,draw,inner sep=0pt,minimum size=1mm] (a2) at (0,0){};
             \node[circle,fill=black,draw,inner sep=0pt,minimum size=1mm] (a3) at (1,0){};
            \draw (x)--(a);
            \draw (x)--(b);
            \draw (x)--(c);
            \draw (x)--(d);
            \draw (x)--(e);
             \draw (x)--(f);
             \draw (a)--(a1);
            \draw (a)--(a2);
             \draw (a)--(a3);
	\end{tikzpicture} 
	 \end{minipage}}
	 \subfigure[$F_5: \rho=2.61313$]{
    \begin{minipage}{0.22\textwidth}
      \centering
  	\begin{tikzpicture}[scale=0.4]
         \node[circle,fill=black,draw,inner sep=0pt,minimum size=1mm] (x) at (2.5,2) {};
            \node[circle,fill=black,draw,inner sep=0pt,minimum size=1mm] (a) at (0,1) {};
            \node[circle,fill=black,draw,inner sep=0pt,minimum size=1mm] (b) at (1,1){};
            \node[circle,fill=black,draw,inner sep=0pt,minimum size=1mm] (c) at (2,1){};
            \node[circle,fill=black,draw,inner sep=0pt,minimum size=1mm] (d) at (3,1){};
            \node[circle,fill=black,draw,inner sep=0pt,minimum size=1mm] (e) at (4,1){};
            \node[circle,fill=black,draw,inner sep=0pt,minimum size=1mm] (f) at (5,1){};
              \node[circle,fill=black,draw,inner sep=0pt,minimum size=1mm] (a1) at (-0.3,0){};
                \node[circle,fill=black,draw,inner sep=0pt,minimum size=1mm] (a2) at (0.3,0){};
             \node[circle,fill=black,draw,inner sep=0pt,minimum size=1mm] (b1) at (0.7,0){};
             \node[circle,fill=black,draw,inner sep=0pt,minimum size=1mm] (b2) at (1.3,0){};
            \draw (x)--(a);
            \draw (x)--(b);
            \draw (x)--(c);
            \draw (x)--(d);
            \draw (x)--(e);
             \draw (x)--(f);
             \draw (a)--(a1);
            \draw (a)--(a2);
             \draw (b)--(b1);
             \draw (b)--(b2);
	\end{tikzpicture} 
	 \end{minipage}}
	  	 \subfigure[$F_6: \rho=2.62386$]{
    \begin{minipage}{0.22\textwidth}
      \centering
  	\begin{tikzpicture}[scale=0.4]
  \node[circle,fill=black,draw,inner sep=0pt,minimum size=1mm] (x) at (2.5,2) {};
            \node[circle,fill=black,draw,inner sep=0pt,minimum size=1mm] (a) at (0,1) {};
            \node[circle,fill=black,draw,inner sep=0pt,minimum size=1mm] (b) at (1,1){};
            \node[circle,fill=black,draw,inner sep=0pt,minimum size=1mm] (c) at (2,1){};
            \node[circle,fill=black,draw,inner sep=0pt,minimum size=1mm] (d) at (3,1){};
            \node[circle,fill=black,draw,inner sep=0pt,minimum size=1mm] (e) at (4,1){};
            \node[circle,fill=black,draw,inner sep=0pt,minimum size=1mm] (f) at (5,1){};
              \node[circle,fill=black,draw,inner sep=0pt,minimum size=1mm] (a1) at (-0.5,0){};
                \node[circle,fill=black,draw,inner sep=0pt,minimum size=1mm] (a2) at (0.5,0){};
  (1.3,0){};
            \draw (x)--(a);
            \draw (x)--(b);
            \draw (x)--(c);
            \draw (x)--(d);
            \draw (x)--(e);
             \draw (x)--(f);
             \draw (a)--(a1);
            \draw (a)--(a2);
           \draw (a1)--(a2);
	\end{tikzpicture} 
	 \end{minipage}}
	 \subfigure[$F_7: \rho=2.59305$]{
    \begin{minipage}{0.22\textwidth}
      \centering
  	\begin{tikzpicture}[scale=0.4]
     \node[circle,fill=black,draw,inner sep=0pt,minimum size=1mm] (x) at (2,2) {};
            \node[circle,fill=black,draw,inner sep=0pt,minimum size=1mm] (a) at (0,1) {};
            \node[circle,fill=black,draw,inner sep=0pt,minimum size=1mm] (b) at (1,1){};
            \node[circle,fill=black,draw,inner sep=0pt,minimum size=1mm] (c) at (2,1){};
            \node[circle,fill=black,draw,inner sep=0pt,minimum size=1mm] (d) at (3,1){};
            \node[circle,fill=black,draw,inner sep=0pt,minimum size=1mm] (e) at (4,1){};
     
              \node[circle,fill=black,draw,inner sep=0pt,minimum size=1mm] (a1) at (-1,0){};
                \node[circle,fill=black,draw,inner sep=0pt,minimum size=1mm] (a2) at (0,0){};
             \node[circle,fill=black,draw,inner sep=0pt,minimum size=1mm] (a3) at (1,0){};
            \draw (x)--(a);
            \draw (x)--(b);
            \draw (x)--(c);
            \draw (x)--(d);
            \draw (x)--(e);
             \draw (a)--(a1);
            \draw (a)--(a2);
             \draw (a)--(a3);
              \draw (a1)--(a2);
	\end{tikzpicture} 
	 \end{minipage}}
	 \subfigure[$F_8: \rho=2.59462$]{
    \begin{minipage}{0.22\textwidth}
      \centering
  	\begin{tikzpicture}[scale=0.4]
      \node[circle,fill=black,draw,inner sep=0pt,minimum size=1mm] (x) at (3.3,2) {};
            \node[circle,fill=black,draw,inner sep=0pt,minimum size=1mm] (a) at (0,2) {};
            \node[circle,fill=black,draw,inner sep=0pt,minimum size=1mm] (b) at (1.8,1){};
            \node[circle,fill=black,draw,inner sep=0pt,minimum size=1mm] (c) at (2.8,1){};
            \node[circle,fill=black,draw,inner sep=0pt,minimum size=1mm] (d) at (3.8,1){};
            \node[circle,fill=black,draw,inner sep=0pt,minimum size=1mm] (e) at (4.8,1){};
              \node[circle,fill=black,draw,inner sep=0pt,minimum size=1mm] (a1) at (-1,1){};
                \node[circle,fill=black,draw,inner sep=0pt,minimum size=1mm] (a2) at (0,1){};
             \node[circle,fill=black,draw,inner sep=0pt,minimum size=1mm] (a3) at (1,1){};
                \node[circle,fill=black,draw,inner sep=0pt,minimum size=1mm] (a11) at (-1.3,0){};
                \node[circle,fill=black,draw,inner sep=0pt,minimum size=1mm] (a12) at (-0.7,0){};
             \node[circle,fill=black,draw,inner sep=0pt,minimum size=1mm] (a21) at (-0.3,0){};
                \node[circle,fill=black,draw,inner sep=0pt,minimum size=1mm] (a22) at (0.3,0){};
                \node[circle,fill=black,draw,inner sep=0pt,minimum size=1mm] (a31) at (0.7,0){};
             \node[circle,fill=black,draw,inner sep=0pt,minimum size=1mm] (a32) at (1.3,0){};
            \draw (x)--(a);
            \draw (x)--(b);
            \draw (x)--(c);
            \draw (x)--(d);
            \draw (x)--(e);
             \draw (a)--(a1);
            \draw (a)--(a2);
             \draw (a)--(a3);   
             \draw (a11)--(a1)--(a12);    
            \draw (a21)--(a2)--(a22);    
             \draw (a31)--(a3)--(a32);    
	\end{tikzpicture} 
	 \end{minipage}}
	 \subfigure[$F_9: \rho=2.59688$]{
    \begin{minipage}{0.22\textwidth}
      \centering
  	\begin{tikzpicture}[scale=0.4]
               \node[circle,fill=black,draw,inner sep=0pt,minimum size=1mm] (x) at (2.8,2) {};
            \node[circle,fill=black,draw,inner sep=0pt,minimum size=1mm] (a) at (0,2) {};
            \node[circle,fill=black,draw,inner sep=0pt,minimum size=1mm] (b) at (1.3,1){};
            \node[circle,fill=black,draw,inner sep=0pt,minimum size=1mm] (c) at (2.3,1){};
            \node[circle,fill=black,draw,inner sep=0pt,minimum size=1mm] (d) at (3.3,1){};
            \node[circle,fill=black,draw,inner sep=0pt,minimum size=1mm] (e) at (4.3,1){};
              \node[circle,fill=black,draw,inner sep=0pt,minimum size=1mm] (a1) at (-0.5,1){};
                \node[circle,fill=black,draw,inner sep=0pt,minimum size=1mm] (a2) at (0.5,1){};
                \node[circle,fill=black,draw,inner sep=0pt,minimum size=1mm] (a11) at (-0.5,0){};
             \node[circle,fill=black,draw,inner sep=0pt,minimum size=1mm] (a21) at (0.5,0){};
            \draw (x)--(a);
            \draw (x)--(b);
            \draw (x)--(c);
            \draw (x)--(d);
            \draw (x)--(e);
             \draw (a)--(a1);
            \draw (a)--(a2);  
             \draw (a2)--(a1);    
            \draw (a11)--(a1);    
            \draw (a21)--(a2);    
	\end{tikzpicture} 
	 \end{minipage}}
	 \subfigure[$F_{10}: \rho=2.56258$]{
    \begin{minipage}{0.22\textwidth}
      \centering
  	\begin{tikzpicture}[scale=0.4]
               \node[circle,fill=black,draw,inner sep=0pt,minimum size=1mm] (x) at (2.3,2) {};
            \node[circle,fill=black,draw,inner sep=0pt,minimum size=1mm] (a) at (0,2) {};
            \node[circle,fill=black,draw,inner sep=0pt,minimum size=1mm] (b) at (1.3,1){};
            \node[circle,fill=black,draw,inner sep=0pt,minimum size=1mm] (c) at (2.3,1){};
            \node[circle,fill=black,draw,inner sep=0pt,minimum size=1mm] (d) at (3.3,1){};
              \node[circle,fill=black,draw,inner sep=0pt,minimum size=1mm] (a1) at (-0.5,1){};
                \node[circle,fill=black,draw,inner sep=0pt,minimum size=1mm] (a2) at (0.5,1){};
                \node[circle,fill=black,draw,inner sep=0pt,minimum size=1mm] (a11) at (-1.5,0){};
             \node[circle,fill=black,draw,inner sep=0pt,minimum size=1mm] (a21) at (0.5,0){};
        \node[circle,fill=black,draw,inner sep=0pt,minimum size=1mm] (a111) at (-1.5,1){};
        \node[circle,fill=black,draw,inner sep=0pt,minimum size=1mm] (a112) at (-2.5,1){};
            \draw (x)--(a);
            \draw (x)--(b);
            \draw (x)--(c);
            \draw (x)--(d);
             \draw (a)--(a1);
            \draw (a)--(a2);  
             \draw (a2)--(a1);    
            \draw (a11)--(a1);    
            \draw (a21)--(a2);    
            \draw (a111)--(a11)--(a112);    
	\end{tikzpicture} 
	 \end{minipage}}
         \subfigure[$F_{11}: \rho=2.59440$]{
    \begin{minipage}{0.22\textwidth}
      \centering
  	\begin{tikzpicture}[scale=0.4]        \node[circle,fill=black,draw,inner sep=0pt,minimum size=1mm] (x) at (2,2) {};
            \node[circle,fill=black,draw,inner sep=0pt,minimum size=1mm] (a) at (1,1) {};
            \node[circle,fill=black,draw,inner sep=0pt,minimum size=1mm] (b) at (3,1){};
              \node[circle,fill=black,draw,inner sep=0pt,minimum size=1mm] (a1) at (0.5,0){};
        \node[circle,fill=black,draw,inner sep=0pt,minimum size=1mm] (a2) at (1,0){};
                \node[circle,fill=black,draw,inner sep=0pt,minimum size=1mm] (a3) at (1.5,0){};
             \node[circle,fill=black,draw,inner sep=0pt,minimum size=1mm] (a4) at (3,0){};
            \draw (x)--(a)--(b)--(x);
            \draw (a)--(a1);
            \draw (a)--(a2);
            \draw (a)--(a3);
             \draw (b)--(a4); 
	\end{tikzpicture} 
	 \end{minipage}}
      \subfigure[$F_{12}: \rho=2.68133$]{
    \begin{minipage}{0.22\textwidth}
      \centering
  	\begin{tikzpicture}[scale=0.4]
       \node[circle,fill=black,draw,inner sep=0pt,minimum size=1mm] (x) at (2,2) {};
            \node[circle,fill=black,draw,inner sep=0pt,minimum size=1mm] (a) at (1,1) {};
            \node[circle,fill=black,draw,inner sep=0pt,minimum size=1mm] (b) at (3,1){};
              \node[circle,fill=black,draw,inner sep=0pt,minimum size=1mm] (a1) at (0.25,0){};
        \node[circle,fill=black,draw,inner sep=0pt,minimum size=1mm] (a2) at (0.75,0){};
                \node[circle,fill=black,draw,inner sep=0pt,minimum size=1mm] (a3) at (1.25,0){};
             \node[circle,fill=black,draw,inner sep=0pt,minimum size=1mm] (a4) at (1.75,0){};
             
            \draw (x)--(a)--(b)--(x);
            \draw (a)--(a1);
            \draw (a)--(a2);
            \draw (a)--(a3);
             \draw (a)--(a4); 
	\end{tikzpicture} 
	 \end{minipage}}
            \subfigure[$F_{13}: \rho=2.58059$]{
    \begin{minipage}{0.22\textwidth}
      \centering
  	\begin{tikzpicture}[scale=0.4]
       \node[circle,fill=black,draw,inner sep=0pt,minimum size=1mm] (x) at (1,2) {};
            \node[circle,fill=black,draw,inner sep=0pt,minimum size=1mm] (a) at (-0.5,1) {};
            \node[circle,fill=black,draw,inner sep=0pt,minimum size=1mm] (b) at (1,1){};
             \node[circle,fill=black,draw,inner sep=0pt,minimum size=1mm] (c) at (2.5,1){};
             \node[circle,fill=black,draw,inner sep=0pt,minimum size=1mm] (a1) at (-1,0){};
        \node[circle,fill=black,draw,inner sep=0pt,minimum size=1mm] (a2) at (0,0){};
                \node[circle,fill=black,draw,inner sep=0pt,minimum size=1mm] (a3) at (1,0){};
             \node[circle,fill=black,draw,inner sep=0pt,minimum size=1mm] (a4) at (2,0){};
             \node[circle,fill=black,draw,inner sep=0pt,minimum size=1mm] (a5) at (3,0){};
            \draw (x)--(a)--(b)--(x);
            \draw (x)--(c);
            \draw (a)--(a1);
            \draw (a)--(a2);
             \draw (b)--(a3); 
             \draw (c)--(a4); 
             \draw (c)--(a5); 
	\end{tikzpicture} 
	 \end{minipage}}
               \subfigure[$F_{14}: \rho=2.59748$]{
    \begin{minipage}{0.22\textwidth}
      \centering
  	\begin{tikzpicture}[scale=0.4]
    \node[circle,fill=black,draw,inner sep=0pt,minimum size=1mm] (x) at (2.5,2) {};
            \node[circle,fill=black,draw,inner sep=0pt,minimum size=1mm] (a) at (0,2) {};
            \node[circle,fill=black,draw,inner sep=0pt,minimum size=1mm] (b) at (2,1){};
            \node[circle,fill=black,draw,inner sep=0pt,minimum size=1mm] (c) at (2.5,1){};
            \node[circle,fill=black,draw,inner sep=0pt,minimum size=1mm] (d) at (3,1){};
              \node[circle,fill=black,draw,inner sep=0pt,minimum size=1mm] (a1) at (-0.5,1){};
                \node[circle,fill=black,draw,inner sep=0pt,minimum size=1mm] (a2) at (0.5,1){};
                \node[circle,fill=black,draw,inner sep=0pt,minimum size=1mm] (a11) at (-1,0){};
             \node[circle,fill=black,draw,inner sep=0pt,minimum size=1mm] (a21) at (1,0){};
        \node[circle,fill=black,draw,inner sep=0pt,minimum size=1mm] (a111) at (-1,1){};
        \node[circle,fill=black,draw,inner sep=0pt,minimum size=1mm] (a112) at (-1.5,1){};
         \node[circle,fill=black,draw,inner sep=0pt,minimum size=1mm] (a211) at (1,1){};
        \node[circle,fill=black,draw,inner sep=0pt,minimum size=1mm] (a212) at (1.5,1){};
            \draw (x)--(a);
            \draw (x)--(b);
            \draw (x)--(c);
            \draw (x)--(d);
             \draw (a)--(a1);
            \draw (a)--(a2);  
             \draw (a2)--(a1);    
            \draw (a11)--(a1);    
            \draw (a21)--(a2);    
            \draw (a111)--(a11)--(a112);  
            \draw (a211)--(a21)--(a212);  
	\end{tikzpicture} 
	 \end{minipage}}
      \subfigure[$F_{15}: \rho=2.57411$]{
    \begin{minipage}{0.22\textwidth}
      \centering
  	\begin{tikzpicture}[scale=0.4]
               \node[circle,fill=black,draw,inner sep=0pt,minimum size=1mm] (x) at (1.5,2) {};
            \node[circle,fill=black,draw,inner sep=0pt,minimum size=1mm] (a) at (0,1) {};
            \node[circle,fill=black,draw,inner sep=0pt,minimum size=1mm] (b) at (1,1){};
            \node[circle,fill=black,draw,inner sep=0pt,minimum size=1mm] (c) at (2,1){};
            \node[circle,fill=black,draw,inner sep=0pt,minimum size=1mm] (d) at (3,1){};
           
              \node[circle,fill=black,draw,inner sep=0pt,minimum size=1mm] (a1) at (0,0){};
                \node[circle,fill=black,draw,inner sep=0pt,minimum size=1mm] (c1) at (1.7,0){};
                \node[circle,fill=black,draw,inner sep=0pt,minimum size=1mm] (c2) at (2.3,0){}; 
                \node[circle,fill=black,draw,inner sep=0pt,minimum size=1mm] (d1) at (2.7,0){};
                \node[circle,fill=black,draw,inner sep=0pt,minimum size=1mm] (d2) at (3.3,0){}; 
            \draw (x)--(a);
            \draw (x)--(b);
            \draw (x)--(c);
            \draw (x)--(d);
             \draw (a)--(b);
            \draw (a)--(a1);  
             \draw (c1)--(c)--(c2);    
             \draw (d1)--(d)--(d2); 
	\end{tikzpicture} 
	 \end{minipage}}
        	 \subfigure[$F_{16}: \rho=2.59462$]{
    \begin{minipage}{0.22\textwidth}
      \centering
  	\begin{tikzpicture}[scale=0.4]
               \node[circle,fill=black,draw,inner sep=0pt,minimum size=1mm] (x) at (2.4,2) {};
            \node[circle,fill=black,draw,inner sep=0pt,minimum size=1mm] (a) at (0,1) {};
            \node[circle,fill=black,draw,inner sep=0pt,minimum size=1mm] (b) at (1.6,1){};
            \node[circle,fill=black,draw,inner sep=0pt,minimum size=1mm] (c) at (3.2,1){};
            \node[circle,fill=black,draw,inner sep=0pt,minimum size=1mm] (d) at (4.8,1){};
           
              \node[circle,fill=black,draw,inner sep=0pt,minimum size=1mm] (a1) at (-0.3,0){};
              \node[circle,fill=black,draw,inner sep=0pt,minimum size=1mm] (a2) at (0.3,0){};
                \node[circle,fill=black,draw,inner sep=0pt,minimum size=1mm] (b1) at (1.3,0){};
              \node[circle,fill=black,draw,inner sep=0pt,minimum size=1mm] (b2) at (1.9,0){};
                \node[circle,fill=black,draw,inner sep=0pt,minimum size=1mm] (c1) at (2.9,0){};
                \node[circle,fill=black,draw,inner sep=0pt,minimum size=1mm] (c2) at (3.5,0){}; 
                \node[circle,fill=black,draw,inner sep=0pt,minimum size=1mm] (d1) at (4.2,0){};
                \node[circle,fill=black,draw,inner sep=0pt,minimum size=1mm] (d2) at (4.6,0){}; 
                 \node[circle,fill=black,draw,inner sep=0pt,minimum size=1mm] (d3) at (5.0,0){}; 
            \node[circle,fill=black,draw,inner sep=0pt,minimum size=1mm] (d4) at (5.4,0){}; 
            \draw (x)--(a);
            \draw (x)--(b);
            \draw (x)--(c);
            \draw (x)--(d); 
              \draw (a1)--(a)--(a2);   
                  \draw (b1)--(b)--(b2);    
             \draw (c1)--(c)--(c2);    
             \draw (d1)--(d)--(d2); 
             \draw (d3)--(d)--(d4); 
	\end{tikzpicture} 
	 \end{minipage}}  
      	 \subfigure[$F_{17}: \rho=\frac{1+\sqrt{17}}{2}$]{
    \begin{minipage}{0.22\textwidth}
      \centering
  	\begin{tikzpicture}[scale=0.4]
               \node[circle,fill=black,draw,inner sep=0pt,minimum size=1mm] (x) at (2.4,2) {};
            \node[circle,fill=black,draw,inner sep=0pt,minimum size=1mm] (a) at (0,1) {};
            \node[circle,fill=black,draw,inner sep=0pt,minimum size=1mm] (b) at (1.6,1){};
            \node[circle,fill=black,draw,inner sep=0pt,minimum size=1mm] (c) at (3.2,1){};
            \node[circle,fill=black,draw,inner sep=0pt,minimum size=1mm] (d) at (4.8,1){};
           
              \node[circle,fill=black,draw,inner sep=0pt,minimum size=1mm] (a1) at (-0.3,0){};
              \node[circle,fill=black,draw,inner sep=0pt,minimum size=1mm] (a2) at (0.3,0){};
                \node[circle,fill=black,draw,inner sep=0pt,minimum size=1mm] (b1) at (1.3,0){};
              \node[circle,fill=black,draw,inner sep=0pt,minimum size=1mm] (b2) at (1.9,0){};
                \node[circle,fill=black,draw,inner sep=0pt,minimum size=1mm] (c1) at (2.7,0){};
                \node[circle,fill=black,draw,inner sep=0pt,minimum size=1mm] (c2) at (3.2,0){};
                \node[circle,fill=black,draw,inner sep=0pt,minimum size=1mm] (c3) at (3.7,0){}; 
                \node[circle,fill=black,draw,inner sep=0pt,minimum size=1mm] (d1) at (4.3,0){};
                \node[circle,fill=black,draw,inner sep=0pt,minimum size=1mm] (d2) at (4.8,0){}; 
                 \node[circle,fill=black,draw,inner sep=0pt,minimum size=1mm] (d3) at (5.3,0){}; 
       
            \draw (x)--(a);
            \draw (x)--(b);
            \draw (x)--(c);
            \draw (x)--(d); 
              \draw (a1)--(a)--(a2);   
                  \draw (b1)--(b)--(b2);    
             \draw (c1)--(c)--(c2);    
             \draw (d1)--(d)--(d2); 
             \draw (d3)--(d); 
              \draw (c3)--(c); 
	\end{tikzpicture} 
	 \end{minipage}}  
           	 \subfigure[$F_{18}: \rho=2.56761$]{
    \begin{minipage}{0.22\textwidth}
      \centering
  	\begin{tikzpicture}[scale=0.4]
               \node[circle,fill=black,draw,inner sep=0pt,minimum size=1mm] (x) at (2.5,2) {};
            \node[circle,fill=black,draw,inner sep=0pt,minimum size=1mm] (a) at (0.2,2) {};
            \node[circle,fill=black,draw,inner sep=0pt,minimum size=1mm] (b) at (1.9,1){};
            \node[circle,fill=black,draw,inner sep=0pt,minimum size=1mm] (c) at (3.1,1){};
            \node[circle,fill=black,draw,inner sep=0pt,minimum size=1mm] (d) at (5.1,2){};
              \node[circle,fill=black,draw,inner sep=0pt,minimum size=1mm] (a1) at (-0.1,1){};
              \node[circle,fill=black,draw,inner sep=0pt,minimum size=1mm] (a2) at (0.5,1){};
                \node[circle,fill=black,draw,inner sep=0pt,minimum size=1mm] (a11) at (-0.4,0){};
              \node[circle,fill=black,draw,inner sep=0pt,minimum size=1mm] (a12) at (0.2,0){};
                \node[circle,fill=black,draw,inner sep=0pt,minimum size=1mm] (b1) at (1.5,0){};
              \node[circle,fill=black,draw,inner sep=0pt,minimum size=1mm] (b2) at (2.3,0){};
               \node[circle,fill=black,draw,inner sep=0pt,minimum size=1mm] (b11) at (1.5,1){};
              \node[circle,fill=black,draw,inner sep=0pt,minimum size=1mm] (b12) at (1.1,1){};
                \node[circle,fill=black,draw,inner sep=0pt,minimum size=1mm] (c1) at (3.5,0){};
                \node[circle,fill=black,draw,inner sep=0pt,minimum size=1mm] (c2) at (2.7,0){};
                \node[circle,fill=black,draw,inner sep=0pt,minimum size=1mm] (c11) at (3.5,1){}; 
                 \node[circle,fill=black,draw,inner sep=0pt,minimum size=1mm] (c12) at (3.9,1){}; 
                \node[circle,fill=black,draw,inner sep=0pt,minimum size=1mm] (d1) at (5.1,1){};
                \node[circle,fill=black,draw,inner sep=0pt,minimum size=1mm] (d2) at (4.5,1){}; 
                \node[circle,fill=black,draw,inner sep=0pt,minimum size=1mm] (d3) at (5.7,1){}; 
            \draw (x)--(a);
            \draw (x)--(b);
            \draw (x)--(c);
            \draw (x)--(d); 
              \draw (a1)--(a)--(a2);   
                  \draw (b1)--(b)--(b2);    
             \draw (c1)--(c)--(c2);    
             \draw (d1)--(d)--(d2); 
      \draw (a11)--(a1)--(a12);   
                  \draw (b11)--(b1)--(b12);    
             \draw (c11)--(c1)--(c12);    
             \draw (d3)--(d); 
	\end{tikzpicture} 
	 \end{minipage}}  
         	 \subfigure[$F_{19}: \rho=2.64575$]{
    \begin{minipage}{0.22\textwidth}
      \centering
  	\begin{tikzpicture}[scale=0.4]
               \node[circle,fill=black,draw,inner sep=0pt,minimum size=1mm] (x) at (2.6,2) {};
            \node[circle,fill=black,draw,inner sep=0pt,minimum size=1mm] (a) at (0,1) {};
            \node[circle,fill=black,draw,inner sep=0pt,minimum size=1mm] (b) at (1.3,1){};
            \node[circle,fill=black,draw,inner sep=0pt,minimum size=1mm] (c) at (2.6,1){};
            \node[circle,fill=black,draw,inner sep=0pt,minimum size=1mm] (d) at (3.9,1){};
             \node[circle,fill=black,draw,inner sep=0pt,minimum size=1mm] (e) at (5.2,1){};
              \node[circle,fill=black,draw,inner sep=0pt,minimum size=1mm] (a1) at (-0.3,0){};
              \node[circle,fill=black,draw,inner sep=0pt,minimum size=1mm] (a2) at (0.3,0){};
                \node[circle,fill=black,draw,inner sep=0pt,minimum size=1mm] (b1) at (1,0){};
              \node[circle,fill=black,draw,inner sep=0pt,minimum size=1mm] (b2) at (1.6,0){};
                \node[circle,fill=black,draw,inner sep=0pt,minimum size=1mm] (c1) at (2.3,0){};
                \node[circle,fill=black,draw,inner sep=0pt,minimum size=1mm] (c2) at (2.9,0){};
                \node[circle,fill=black,draw,inner sep=0pt,minimum size=1mm] (d1) at (3.6,0){};
                \node[circle,fill=black,draw,inner sep=0pt,minimum size=1mm] (d2) at (4.2,0){}; 
                 \node[circle,fill=black,draw,inner sep=0pt,minimum size=1mm] (e1) at (4.9,0){}; 
           \node[circle,fill=black,draw,inner sep=0pt,minimum size=1mm] (e2) at (5.5,0){}; 
            \draw (x)--(a);
            \draw (x)--(b);
            \draw (x)--(c);
            \draw (x)--(d); 
            \draw (x)--(e); 
              \draw (a1)--(a)--(a2);   
                  \draw (b1)--(b)--(b2);    
             \draw (c1)--(c)--(c2);    
             \draw (d1)--(d)--(d2); 
 \draw (e1)--(e)--(e2); 
	\end{tikzpicture} 
	 \end{minipage}}  
       	 \subfigure[$F_{20}: \rho=2.58605$]{
    \begin{minipage}{0.22\textwidth}
      \centering
  	\begin{tikzpicture}[scale=0.4]
                   \node[circle,fill=black,draw,inner sep=0pt,minimum size=1mm] (x) at (2,2) {};
            \node[circle,fill=black,draw,inner sep=0pt,minimum size=1mm] (a) at (0,1) {};
            \node[circle,fill=black,draw,inner sep=0pt,minimum size=1mm] (b) at (1,1){};
            \node[circle,fill=black,draw,inner sep=0pt,minimum size=1mm] (c) at (2,1){};
            \node[circle,fill=black,draw,inner sep=0pt,minimum size=1mm] (d) at (3,1){};
             \node[circle,fill=black,draw,inner sep=0pt,minimum size=1mm] (e) at (4,1){};
       
                \node[circle,fill=black,draw,inner sep=0pt,minimum size=1mm] (c1) at (1.7,0){};
                \node[circle,fill=black,draw,inner sep=0pt,minimum size=1mm] (c2) at (2.3,0){};
        
            \draw (x)--(a);
            \draw (x)--(b);
            \draw (x)--(c);
            \draw (x)--(d); 
            \draw (x)--(e); 
            \draw (a)--(b);   
       \draw (c1)--(c)--(c2);    
	\end{tikzpicture} 
	 \end{minipage}}  
	 \caption{The graphs with $\rho(G)\geq \frac{1+\sqrt{17}}{2}$}
	 \label{graphsF}
\end{figure}
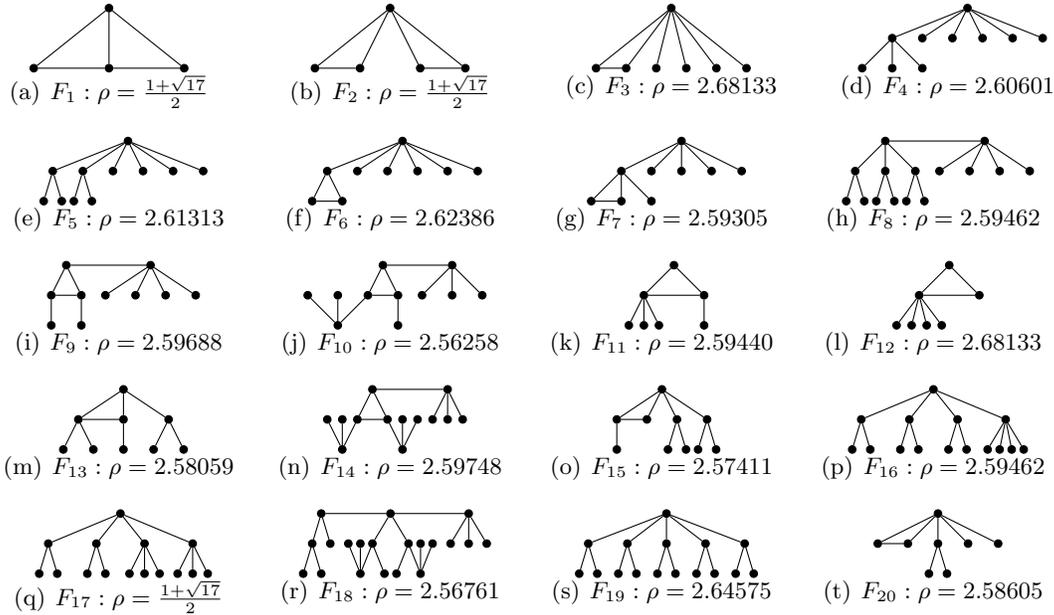

	\subsection*{Acknowledgements}
        Junxue Zhang was supported by the China Postdoctoral Science Foundation (No. 2024M764113).

\end{document}